\documentclass[12pt,reqno]{amsart}
\usepackage{amsmath, amsthm, amsopn, amssymb, microtype}

\usepackage{mathrsfs} 
\usepackage{mathscinet}
\usepackage{bbm}
\usepackage{enumerate, tikz, tikz-cd, etoolbox, intcalc, geometry, caption, subcaption, afterpage, enumitem}
\usepackage[english]{babel}
\usepackage{tkz-graph}
\usetikzlibrary{arrows}
\usepackage{caption}
\usepackage{bbm}
\title{Rationality and computability of the covering radius for sofic shifts}
\author{Tom Meyerovitch}

\author{Aidan Young}
\email{mtom@bgu.ac.il, aidanyoungmath@gmail.com}
\address{Ben Gurion University of the Negev.
	Departement of Mathematics.
	Be'er Sheva, 8410501, Israel
}

\usepackage[colorlinks=true,urlcolor=blue,pdfborder={0 0 0}]{hyperref}
\hypersetup{linkcolor=[rgb]{0,0,0.6}}
\hypersetup{citecolor=[rgb]{0,0.6,0}}
\usepackage[nameinlink]{cleveref}
\crefname{theorem}{Theorem}{Theorems}
\crefname{thm}{Theorem}{Theorems}
\crefname{mainthm}{Theorem}{Theorems}
\crefname{lemma}{Lemma}{Lemmas}
\crefname{lem}{Lemma}{Lemmas}
\crefname{remark}{Remark}{Remarks}
\crefname{prop}{Proposition}{Propositions}
\crefname{defn}{Definition}{Definitions}
\crefname{cor}{Corollary}{Corollaries}
\crefname{section}{Section}{Sections}
\crefname{figure}{Figure}{Figures}
\crefname{quest}{Question}{Questions}
\crefname{notation}{Notation}{Notations}
\crefname{conv}{Convention}{Conventions}
\crefname{example}{Example}{Examples}

\begin{document}

\theoremstyle{plain}
\newtheorem{thm}{Theorem}
\newtheorem{lemma}{Lemma}[section]
\newtheorem{prop}[lemma]{Proposition}
\newtheorem{cor}[lemma]{Corollary}
\newtheorem{claim}[lemma]{Claim}
\newtheorem{quest}[lemma]{Question}
\newtheorem{fact}[lemma]{Fact}

\renewcommand*{\thethm}{\Alph{thm}}

\theoremstyle{definition}
\newtheorem{definition}[lemma]{Definition}
\newtheorem{example}[lemma]{Example}
\newtheorem{remark}[lemma]{Remark}
\newtheorem{conj}[lemma]{Conjecture}
\newtheorem{notation}[lemma]{Notation}
\newtheorem{conv}[lemma]{Convention}

\newcommand{\N}{\mathbb{N}}
\newcommand{\Z}{\mathbb{Z}}
\newcommand{\Q}{\mathbb{Q}}
\newcommand{\R}{\mathbb{R}}
\newcommand{\C}{\mathbf{C}}
\newcommand{\cL}{\mathcal{L}}
\newcommand{{\cP}}{\mathcal{P}}
\newcommand{\Pna}{\mathcal{P}^\Omega}
\newcommand{\oP}{\overline{P}}
\newcommand{\cS}{\mathcal{S}}
\newcommand{\cT}{\mathcal{T}}  
\newcommand{\cV}[1]{\mathcal{V}_{#1}}
\newcommand{\cE}[1]{\mathcal{E}_{#1}}
\newcommand{\Aval}{\mathbb{V}_\mathit{alt}}
\newcommand{\NAval}{\mathbb{V}_\mathit{non-alt}}
\newcommand{\per}{\operatorname{per}}
\newcommand{\cW}{\mathcal{W}}
\newcommand{\NI}{\mathbf{NI}\;}
\newcommand{\cM}{\mathcal{M}}
\newcommand{\sW}{\mathscr{W}}
\newcommand{\cR}{\mathcal{R}}

\newcommand{\SV}{\operatorname{SV}}
\newcommand{\Var}{\operatorname{Var}}
\newcommand{\MaxEavg}{\operatorname{MaxAvg}}
\newcommand{\Prob}{\operatorname{Prob}}

\begin{abstract}
The covering radius of a shift space is a quantity of interest for  information-theoretic applications of data transmission over noisy channels. 
We prove that the covering radius of a primitive sofic shift is a rational number, and describe an algorithm to compute the covering radius from a labeled graph presentation. 
\end{abstract}

\maketitle
\section{Introduction}\label{sec:intro}

The covering radius of a code $\mathcal{C} \subseteq \{0,1\}^n$ is the smallest number $R>0$ such that for any $w \in \{0,1\}^n$, there exists $v \in \mathcal{C}$ that differs from $w$ in at most $R$ coordinates.  
We refer the reader to the book \cite{MR1453577} for an extensive account of covering radius of codes and a wide range of applications related to this notion.

The covering radius of a shift space $S \subseteq \{0,1\}^\Z$ is the natural analog for infinite sequences. We recall the formal definition in \Cref{subsec:covering_redius}. Covering radii of shift spaces have been studied in \cite{MR4751633} as a key parameter framework for implementing error-correction in constrained systems. 

In \cite{MR4751633}, the covering radius has been explicitly computed for some fundamental examples of sofic shifts, including run length-limited shifts with certain parameters. Sofic shifts, and run length-limited shifts in particular,  occur in the context of constrained coding. Sofic shifts (see \Cref{subsec:sofic_shifts} for a definition) constitute an important family that come up in real-word data transmission and storage applications. The computations of the covering radius in \cite{MR4751633} were preformed on a somewhat ad-hoc basis. In all the examples, the covering radius turned out to be a rational number. The following questions naturally occurred: Is the covering radius of a sofic shift always a rational number? Is there a general algorithm to compute the covering radius of a sofic shift?

Here we answer these questions in the fundamental case that the underlying sofic shift is primitive.

\begin{thm}\label{thm:covering_radius_rational}
Let $\mathcal{G}=(G,L)$ be a primitive labeled graph.
Then the covering radius $R(X_\mathcal{G})$ is a rational number.
\end{thm}

\begin{thm}\label{thm:covering_radius_compute}
There exists an algorithm that receives as input a primitive labeled graph $\mathcal{G}$ and computes the covering radius $R(X_\mathcal{G})$ in finite time.
\end{thm}

In the above, $X_{\mathcal{G}}$ is the edge shift corresponding to the labeled graph $\mathcal{G}$. We will recall the definition of an edge shift in the following section, specifically in  \eqref{eq:X_sofic}.

We obtain these results as part of a more general framework that can be expressed as a certain two-player game, as in \cite{meyerovitch2025nonalternatingmeanpayoffgames}.
The proof methods and results presented in this paper are rather elementary and combinatorial in nature.

The paper is organized as follows: In \Cref{sec:prelim}, we quickly recall basic concepts and terminology and notation about words, labeled graphs and shift spaces. We also recall formally the definition of a shift space. In \Cref{sec:tropical_conv}, we present a certain binary operation on integer-valued bi-variate functions which we call ``tropical convolution" that will be used in the proof of our main results.
In \Cref{sec:mean_payoff_game}, we present a variation on the mean payoff games introduced by A. Ehrenfeucht and J. Mycielski in \cite{EhrenfeuchtMycielski}, and deduce \Cref{thm:covering_radius_rational,thm:covering_radius_compute} from \Cref{thm:game_value_rational_computable}, a result about these modified mean payoff games.

In \Cref{sec:non_improvable_paths}, we consider shift spaces
of ``optimal strategies'' for the mean payoff games introduced in \Cref{sec:mean_payoff_game}. We prove that these shift spaces are sofic shifts. 
Finally, in \Cref{sec:concluding_remarks}, we conclude with further remarks and open problems. 

\noindent\textbf{Acknowledgment:} This research was supported by Israel Science Foundation grant no. 985/23.

\section{Preliminaries}\label{sec:prelim}
\subsection{Words and codes}
Let $\Sigma$ be a finite set of ``symbols'', for instance $\Sigma=\{0,1\}$.
We refer to $\Sigma$ as the ``alphabet'', and refer to $n$-tuples in $\Sigma^n$ as ``words" of length $n$ over the alphabet $\Sigma$.

A \emph{code} is a set of words $\mathcal{C} \subseteq \Sigma^n$,
and the \emph{length} of the code is the number $n$.

\subsection{Labeled graphs}
In this paper by a \emph{graph}, we always mean a \emph{finite, directed} graph, namely a pair $G=(\cV{G},\cE{G})$, where $\cV{G}$ is the finite set of vertices and $\cE{G}$ is the finite set of (directed) edges. Every edge $e \in \cE{G}$ has a \emph{source} $s(e) \in \cV{G}$ and a \emph{target} $t(e) \in \cV{G}$. 
We allow parallel edges, i.e. it is possible that two vertices have more than one edge with a common source and common target. We also allow self-loops, i.e. edges such that $s(e)=t(e)$.
A \emph{walk} of length $n$ in the graph $G$ is a tuple $(e_1,\ldots,e_n) \in \cE{G}^n$ such that $t(e_k)=s(e_{k+1})$ for every $1\le k <n$. Given $v,w \in \cV{G}$ we denote by $\cW_n(G, v \to w)$ the  the set of length $n$ walks in $G$ that start at the vertex $v$ and terminate at the vertex $w$:
\begin{equation}
    \cW_n(G, v \to w) := \left\{
    (e_j)_{j = 1}^n \in \cE{G}^n ~:~ s(e_1) = v, t(e_n)=w,~ t(e_{i-1})=s(e_{i}) ~\forall 1 < i  \le n
    \right\}.
\end{equation}
We also denote
\begin{align*}
\cW_n(G,v)  & := \bigcup_{w \in \cV{G}}\cW_n(G,v \to w) ,   & \cW_n(G)  & := \bigcup_{v \in \cV{G}}\cW_n(G,v),
\end{align*}
and
\[
\cW_{\infty}(G) := \left\{ (e_n)_{n \in \N} \in \cE{G} ~:~ t(e_n)=s(e_{n+1}) ~ \forall n \in \N \right\}.
\]

Given $q \in \cW_n(G)$ and $1\le i \le n$ we denote by $q_i \in \cE{G}$ the $i$th edge in $q = (q_1, \ldots, q_n)$. 

Given $n,m \in \N$; $u,v,w \in \cV{G}$; $p \in \cW_n(G,u \to v)$; and $q \in \cW_m(G,v \to w)$, the \emph{concatenation} which we denote by $p * q \in \cW_{n+m}(G,u \to w)$. is given by
\[
(p*q)_i =\begin{cases}
    p_i & \textrm{if $1\le i \le n$,}\\
    q_{i-n} & \textrm{if $n+ 1\le i \le n+m$.}
\end{cases}
\]

A directed graph $G = (\cV{G}, \cE{G})$ is called \emph{irreducible} if any two vertices are connected by a walk. 

A \emph{cycle}
 in $G$ is a walk $(e_1,\ldots,e_n)$ such that $s(e_1)=t(e_n)$. Note that we do \emph{not} assume that a cycle satisfies any injectivity condition. An irreducible
graph is called \emph{primitive} if the greatest common divisor of all cycle lengths is 1.

A labeled graph is a pair $\mathcal{G}=(G,L)$, where $G$ is a graph and $L:\cE{G} \to \Sigma$ is a \emph{labeling function} on the edges of $G$, where $\Sigma$ is a finite set.
Given  $p \in \cW_{n}(G)$ we denote
\[
L(p) := (L(p_1),\ldots,L(p_n)) \in \Sigma^n.
\]
Given a labeled graph $\mathcal{G}=(G,L)$, we denote by $\mathcal{C}_n(\mathcal{G}) \subseteq \Sigma^n$ the words of length $n$ that occur as a the labels of some walk of length $n$ on the  underlying graph $G$:
\[
\mathcal{C}_n(\mathcal{G}) := \left\{ (L(p) \in \Sigma^n ~:~ p \in \cW_n(G) \right\}.
\]

We say that $\mathcal{G}=(G,L)$ is irreducible (primitive) when the underlying graph $G$ is irreducible (primitive). 

\subsection{Symbolic dynamics: Shift spaces  and sofic shifts}\label{subsec:sofic_shifts}
Let $\Sigma$ be a finite set of ``symbols'', for instance $\Sigma=\{0,1\}$.
We refer to $\Sigma$ as the ``alphabet'', and refer  to tuples in $\Sigma^n$ as \emph{words of length $n$} over the alphabet $\Sigma$. We denote the set of words over $\Sigma$ by $\Sigma^* := \bigcup_{n=0}^\infty \Sigma^n$. 

Considering $\Sigma$ as a finite discrete topological space, the space $\Sigma^\Z$ of $\Sigma$-valued functions on $\Z$ can be equipped with the product topology, which makes it a compact metrizable topological space.
It is convenient to think of the elements of $\Sigma^\Z$ as bi-infinite sequences with values in $\Sigma$.
Given $x \in \Sigma^\Z$ and $m,n \in \Z$ with $m \le n$ we denote
\[
x_{[m,n]}:= (x_m,\ldots,x_n) \in \Sigma^{m-n+1}.
\]

The \emph{shift map} is the function  $\sigma:\Sigma^\Z\to \Sigma^\Z$ defined by 
\[
\sigma((x_n)_{n \in \Z})= (x_{n+1})_{n\in \Z}.
\]
Then $\sigma:\Sigma^\Z\to \Sigma^\Z$ is a homeomorphism.
A \emph{shift space} or \emph{subshift} is a closed, $\sigma$-invariant subset $S \subseteq \Sigma^\Z$.

Given a shift space $S$ and $n \in \N$ we denote 
by $\mathcal{C}_n(S)$ the set of admissible words of length $n$ for $S$. That is:
\begin{equation*}\label{eq:C_n_def}
    \mathcal{C}_n(S) := \left\{ x_{[1,n]} ~:~ x \in S\right\} \subseteq \Sigma^n.
\end{equation*}

Also denote
\[
\mathcal{C}(S) := \bigcup_{n \in \N}\mathcal{C}_n(S).
\]

\begin{definition}\label{def:generator_of_shift}
Consider a shift space $S \subseteq \Sigma^\Z$. We say that a set $\mathcal{F} \subseteq \Sigma^*$ \emph{generates} $S$ if
$$
    S = \left\{x \in \Sigma^\Z ~:~  x_{[m,n]} \not \in \mathcal{F}~ \forall m\le n  \right\}.
$$
\end{definition}

It is well-known (c.f. \cite[Theorem 6.1.21]{LindMarcus}) that a set $S \subseteq \Sigma ^\Z$ is a shift space if and only if there exist $\mathcal{F} \subseteq \Sigma^*$ that generates $S$, although $\mathcal{F}$ is not in general unique. A shift space $S \subseteq \Sigma^\Z$ that can be generated by a \emph{finite} set $\mathcal{F} \subset \Sigma^*$
is called a \emph{shift of finite type (SFT)}.

A shift space $S \subseteq \Sigma^\Z$ is called \emph{sofic} if there exists a labeled graph $\mathcal{G}=(G,L)$ such that $S$ is set bi-infinite sequences that are obtained by reading the $L$-labels of a bi-infinite walk in $G$. We call such a labeled graph $\mathcal{G}$ a \emph{labeled graph presentation} of $S$. Equivalently, a sofic shift is a shift space of the form $X_\mathcal{G}$ where 
\begin{equation}\label{eq:X_sofic}
    X_\mathcal{G} = \left\{ x \in \Sigma^\Z~:~  x_{[m,n]} \in \mathcal{C}_{n-m+1}(\mathcal{G}) ~ \forall m \le n \right\}.
\end{equation}

It is well known that any shift of finite type is a sofic shift. For details and discussion of sofic shifts and their importance, see \cite[Chapter 3]{LindMarcus}.  
In particular, sofic shifts are a natural model for information storage and transmission.

A sofic shift $X$ is irreducible (primitive) if there exists an irreducible (primitive) labeled graph $\mathcal{G}$ such that $X=X_\mathcal{G}$. We note that there are equivalent ``intrinsic'' definitions, not involving a representation by a labeled graph. See for instance \cite[Chapter 6]{LindMarcus}.

\subsection{The covering radius of shift space}\label{subsec:covering_redius}

The \emph{hamming distance} between two words $u,v \in \Sigma^n$ is the number of coordinates in which $u$ differs from $v$:
\[
d_H(u,v) : = |\{1\le i \le n ~:~ u_i \ne v_i \}|=\sum_{i=1}^n \delta(u_i,v_i),
\]
where
\[
\delta(a,b) = \begin{cases}
    1 & \mbox{ if } a=b\\
    0 & \mbox{otherwise.}
\end{cases}
\]
The \emph{covering radius} of a code $\mathcal{C} \subseteq \Sigma^n $ is given by
\[
R(\mathcal{C}):= \max_{v \in \Sigma^n}\min_{w \in \mathcal{C}}d_H(v,w).
\]

Let $X \subseteq \Sigma^\Z$ be a shift space. Following  \cite{MR4751633}, the \emph{(lower/upper) covering radius} of $X$ are defined respectively by
\begin{equation*}\label{eq:covering_radius_def}
    \underline{R}(X) = \liminf_{n \to \infty}\frac{1}{n}R(\mathcal{C}_n(X)) \mbox{ and }
    \overline{R}(X) = \limsup_{n \to \infty}\frac{1}{n}R(\mathcal{C}_n(X)).
\end{equation*}
Whenever $\overline{R}(X)=\underline{R}(X)$, their common value is called the \emph{covering radius} of $X$. When the limit exists we will denote the covering radius by 
\[
R(X) = \lim_{n \to \infty}\frac{1}{n}R(\mathcal{C}_n(X)).
\]

We define the \emph{hamming pseudo-metric} on $\Sigma^\Z$ by
\begin{equation*}\label{eq:overline_d_def}
\overline{d}_H(x,y)= \limsup_{n\to \infty}\frac{1}{n}d_H((x_1,\ldots,x_n),(y_1,\ldots,y_n)), ~ x,y \in \Sigma^\Z.
\end{equation*}

Given $y \in \Sigma^\Z$ and $X \subseteq \Sigma^\Z$ 
define
\begin{equation*}\label{eq:overline_d_X_def}
\overline{d}_H(y,X) := \inf_{x \in X}\overline{d}_H(y,x).
\end{equation*}
Similarly, define
\begin{equation*}\label{eq:underline_d_def}
\underline{d}_H(x,y) : = \liminf_{n \to \infty} \frac{1}{n} d_H((x_1,\ldots,x_n),(y_1,\ldots,y_n)), ~ x,y \in \Sigma^\Z   
\end{equation*}
and 
\[
\underline{d}_H(y,X) = \inf_{x \in X}\underline{d}_H(y,x).
\]
It is easy to see from the definition of the (upper) covering radius $R(X)$ that for any subshift $X \subseteq \Sigma^\Z$ and any $y \in \Sigma^\Z$ 
\[
\overline{d}_H(y,X) \le \overline{R}(X). 
\]

It was shown in \cite[Proposition 7]{MR4751633} that the limit $R(\mathcal{G}):= R(X_\mathcal{G})=\lim_{n \to \infty} \frac{1}{n}R(\mathcal{C}_n(\mathcal{G}))$ exists for any labeled graph $\mathcal{G}$. In fact, the statement of  \cite[Proposition 7]{MR4751633} is slightly more general.

\section{Tropical convolution and maximal functions}\label{sec:tropical_conv}
In this section, we present and discuss a binary operation called tropical convolution on the space $\Z_{\infty}^{E \times E}$ of $\Z_{\infty} : = \Z \cup \{+\infty\}$-valued functions on a product space $E \times E$, where $E$ is an arbitrary \emph{finite} set. This operation will be important for proving \Cref{thm:game_value_rational_computable}, which in turn will help us prove \Cref{thm:covering_radius_compute,thm:covering_radius_rational}.

Given two functions $f,g \in \Z_\infty^{E \times E}$ define the \emph{tropical convolution} $f * g$ by
\[
(f*g)(u,w)= \min_{v \in E}\left(f(u,v)+g(v,w)\right),~ u,w \in E.
\]
A routine verification shows that the tropical convolution is an associative operation in the sense that $(f*g)*h=f*(g*h)$, so we will freely use expressions such as $f*g*h$. Given subsets $A,B \subseteq \Z_\infty^{E \times E}$, we use the notation
\[
A*B : = \{ f*g ~:~ f \in A,~ g\in B\}.
\]
Tropical convolution can be seen, as the name would suggest, as a tropical analog of the more classical notion of convolution, where we replace the classical ring $(\Z, +, \cdot)$ with the tropical minplus semiring $(\Z_{\infty}, \min(\cdot, \cdot), +)$. This idea of tropical convolution has been used previously in the construction of neural networks, as in \cite{TropicalConvolutionNeuralNetworks}.

We define a partial order on the set $\Z_\infty^{E \times E}$ by
\[
f < g ~\Longleftrightarrow ~ \forall v,w \in E\, \left( f(v,w) < g(v,w) \right).
\]

Given $A \subseteq \Z_\infty^{E \times E}$, we say that $f \in A$ is \emph{$<$-maximal} in $A$ if there does not exists $g \in A$ such that $f < g$. 
We denote the set of $<$-maximal elements in $A$ by $A^\dagger$.

\begin{lemma}\label{lem:maximal_factorization}
    Let $A,B \subset \Z_\infty^{E \times E}$ be finite sets. If $f \in A$, $g \in B$ are such that $f* g \in (A*B)^\dagger$ and $f*g$ is $\Z$-valued (i.e. does not attain the value $+ \infty$), then $f \in A^\dagger$ and $g \in B^\dagger$.
\end{lemma}
\begin{proof}
We prove that $f$ is $<$-maximal in $A$. If $f$ takes the value $+ \infty$, then $f$ is trivially $<$-maximal, so consider the case where $f$ is $\Z$-valued. If $f$ is $\Z$-valued and not $<$-maximal in $A$, then there exists $f' \in A$ with $f' > f$. Since $f*g$ is assumed to be $\Z$-valued, we can write
\begin{align*}
(f*g)(u, w)   & = \min_{\substack{v \in E, \\
g(v, w) \in \Z}} \left[ f(u, v) + g(v, w) \right] \\
    & < \min_{\substack{v \in E, \\
g(v, w) \in \Z}} \left[ f'(u, v) + g(v, w) \right] \\
    & \leq \min_{v \in E} \left[ f'(u, v) + g(v, w) \right] & = \left( f'*g \right) (u, w)  & & \textrm{for all $u, w \in E$.}
\end{align*}
Thus $f * g < f'*g$, contradicting the $<$-maximality of $f*g$ in $A*B$.

An analogous argument proves that $g$ must be $<$-maximal in $B$.
\end{proof}

\begin{lemma}\label{lem:star_trop_conv}
    For any  finite sets $A,B \subset \Z^{E \times E}$ we have 
    \[
    (A^\dagger * B^\dagger)^\dagger = (A*B)^\dagger.
    \]
\end{lemma}
\begin{proof}
    From \Cref{lem:maximal_factorization}, it follows that
    $(A*B)^\dagger$ consists exactly of the elements of $A^\dagger * B^\dagger$ that are $<$-maximal in $A*B$.
    In particular, any element of $(A*B)^\dagger$ is a $<$-maximal element of $A^\dagger*B^\dagger$, so 
    \[
    (A*B)^\dagger \subseteq (A^\dagger * B^\dagger)^\dagger.
    \]
    
    To complete the proof, we need to check that any element  $f*g \in A^\dagger * B^\dagger$ that is $<$-maximal in $A^\dagger * B^\dagger$ is also $<$-maximal in $A*B$; we prove the contrapositive, that if $f * g$ is if not maximal in $A * B$, then $f* g$ is not maximal in $A^\dagger * B^\dagger$.
    By the assumption that $A$ and $B$ are finite, if $f*g \in A*B$ is not $<$-maximal in $A*B$, then there exists $f'*g' \in A*B$ which is $<$-maximal in $A*B$ and $f*g < f'*g'$ (otherwise we get an infinite $<$-chain in $A*B$, contradicting finiteness).
    By the earlier argument, if $f'*g' \in A*B$ is $<$-maximal in $A*B$, then $f' \in A^\dagger$ and $g' \in B^\dagger$, meaning $f * g$ is not maximal in $A^\dagger * B^\dagger$. By contraposition, it follows that any element of $(A^\dagger * B^\dagger)^\dagger$ is also in $(A*B)^\dagger$.
\end{proof}

For $n \in \Z$ and $f \in \Z^{E \times E}$, let $f+n\in  \Z^{E \times E}$ be the function given by pointwise incrementing $f$ by $n$.
The direct unraveling of the definitions show that for any $f,g \in \Z^{E \times E}$ and $n \in \Z$ we have:
\begin{equation}\label{eq:conv_plus_const}
    (f+n)*g = f*(g+n) = (f* g) +n.
\end{equation}

Given $A \subseteq \Z^{E \times E}$ and $n \in \Z$, set 
\[
A+n : = \{f+n~:~ f \in A \}.
\]
Then \eqref{eq:conv_plus_const} translates to the fact that for any $A,B \subseteq \Z^{E \times E}$ and $n \in \Z$ the following holds:
\begin{equation}\label{eq:set_conv_plus_const}
(A+n)*B=A*(B+n)=(A*B) + n.    
\end{equation}

For $f \in \Z^{E \times E}$, we denote
\begin{align*}
\|f\|   &:=\max_{u,v \in E}|f(u,v)|, &
\max f & := \max_{u,v \in E}f(u,v), \\
\min f  & := \min_{u,v \in E}f(u,v) , &
\Delta f &  :=\max f -\min f .
\end{align*}

\begin{lemma}\label{lem:3_prodcuct_delta}
    Suppose that  $f_1,f_2,f_3 \in \Z^{E\times E}$ and $f:= f_1 * f_2 * f_3$.
    Then 
    \begin{equation}\label{eq:Delta_f_bound}
        \Delta f \le 2\|f_1\| + 2\|f_3\|.
    \end{equation}
\end{lemma}
\begin{proof}
        For any finite set $X$ and functions $\phi,\psi:X \to \mathbb{R}$ we have 
    \[\min_{x \in X}\phi(x)-\min_{x \in X}\psi(x) = \min_{x \in X}\max_{y \in Y}(\phi(x)-\psi(y)) \le \min_{x \in X}(\phi(x)-\psi(x)).\] 
    For any fixed $v,w,v',w' \in E$, applying the above formula on the functions $\phi,\psi:E\times E\to \mathbb{R}$ given by
    
    \[\phi(u,u'):=f_1(v,u)+f_2(u,u')+f_3(u',w)\]
    and
    \[\psi(u,u')=f_1(v',u)+f_2(u,u')+f_3(u',w'),\]
    we obtain the following inequality:
    \[
     f(v,w) -f(v',w') \le \min_{u,u' \in \cV{G}}( f_1(v,u) - f_1(v',u)+f_3(u',w)-f_3(u',w'))
    \]
    Thus,
    \[
     f(v,w) -f(v',w') \le 2\|f_1\| + 2\|f_2\|.
    \]
    Taking maximum over $(v,w) \in E \times E$ and minimum over $(v',w') \in E \times E$ we get \eqref{eq:Delta_f_bound}.
\end{proof}
Given a finite set $A \subseteq \Z^{E\times E}$, set 
\[
\Delta A := \max_{f \in A}\max f -\min_{g \in A}\min g.
\]

\section{A zero sum two player mean payoff game}\label{sec:mean_payoff_game}
In this section we derive the main results, namely \Cref{thm:covering_radius_compute} and \Cref{thm:covering_radius_rational} as a special case of a slightly more general result. Underlying this more general result are certain types of zero-sum two-player game. We introduced and studied these games in \cite{meyerovitch2025nonalternatingmeanpayoffgames} and called them ``non-alternating mean-payoff games." As mentioned in \cite{meyerovitch2025nonalternatingmeanpayoffgames}, these games can be considered as a variant or analog of the much more classical ``alternating'' mean-payoff games introduced in \cite{EhrenfeuchtMycielski}. We will quickly recall the rules of these games.

Let $G=(\cV{G},\cE{G})$ and $H=(\cV{H},\cE{H})$ be two graphs, and let $P:\cE{G} \times \cE{H} \to \Z$ be a function which we will refer to as ``the payoff function." For the rest of this section we will assume that $G$ and $H$ are primitive. The game takes place between two players whom we name Alice and Bob. The $n$-turn game goes as follows: First, Alice will choose a walk $p \in \cW_n(H)$ of length $n$, after which Bob will choose (with full knowledge of Alice's play) a walk $q \in \cW_n(H)$ of length $n$; the game will then conclude, and Bob will pay Alice $\sum_{i = 1}^n P(q_i, p_i)$ pesos.

Given $n \in \N$, we define 
\begin{equation*}\label{eq:V_n_def}
V_n(G,H,P) := \max_{p \in \cW_n(H)}\min_{q \in \cW_n(G)}\sum_{i=1}^n P(q_i,p_i)
\end{equation*}
and
\begin{equation}\label{eq:V_def}
    V(G,H,P) := \lim_{n \to \infty}\frac{1}{n}V_n(G,H,P).
\end{equation}
It is known that the limit in \eqref{eq:V_def} exists, see for instance
\cite[Proposition 2.8]{young2025adversarialergodicoptimization} or \cite[Theorem 3.1]{meyerovitch2025nonalternatingmeanpayoffgames}. The fact that the limit exists will will also follow from \Cref{lem:M_star_periodic} later in this article. We can think of this as the value of a kind of infinite version of the game.

Our main result in this section is the following:
\begin{thm}\label{thm:game_value_rational_computable}
    Let $G$ and $H$ be primitive graphs, and let $P:\cE{G} \times \cE{H} \to \Z$ be a payoff function. Then $V(G,H,P)$ is a rational number.
    Furthermore, there exists an algorithm that computes $V(G,H,P)$ from $G,H,P$ in finite time.
\end{thm}

From \Cref{thm:game_value_rational_computable} we can deduce \Cref{thm:covering_radius_rational} and \Cref{thm:covering_radius_compute} directly as follows:

\begin{proof}[Proof of \Cref{thm:covering_radius_rational} and \Cref{thm:covering_radius_compute}, assuming \Cref{thm:game_value_rational_computable}]
Let $S=X_\mathcal{G}$ be a primitive sofic shift, where  $\mathcal{G}=(G,L)$ is a primitive labeled graph presentation of $S$ and $L: \cE{G} \to \Sigma$ is a labeling function. Define $H$ to be the graph consisting of a unique vertex $v_0$ with $|\Sigma|$ many self-loops:
\begin{equation*}\label{eq:H_aux_def}
\cV{H}=\{v_0\},~ \cE{H} = \Sigma,~ 
s(a)=t(a)=v_0 \mbox{ for all } a \in \Sigma.
\end{equation*}
$\cV{H}=\{v_0\}$, and $\cE{H}=\Sigma$ such that $s(a)=t(a)=v_0$ for every $a \in \Sigma$. Define $P:\cE{G} \times \cE{H} \to \Z$ as
\begin{align*}\label{align:P_reduction_def}
P(e,a)  & = \begin{cases}
    0 &  \mbox{ if } L(e)=a\\
    1 & \mbox{ otherwise}
\end{cases} & \textrm{for $e \in \cE{G}, a \in \Sigma = \cE{H}$.}
\end{align*}
Then because $H$ has a unique vertex $v_0$, for every $n \in \N$ $\cW_n(H)= \cW_n(H, v_0 \to v_0)= \Sigma^n$.

Then the definition of $P$ and $H$ implies that for every $p \in \cW_{n}(G)$ and $q \in \cW_n(H) = \Sigma^n$, we have
\[
\sum_{j=1}^n P(p_i,q_i)= d_H(L(p),w).
\]
It follows directly that 
$V_n(G,H,P)=R(\mathcal{C}_n(\mathcal{G}))$. In particular, with $H$ and $P$ as above, we have
\[
R(X_\mathcal{G})=V(G,H,P).
\]
By the first part of \Cref{thm:game_value_rational_computable},
$V(G,H,P)$ is a rational number so \Cref{thm:covering_radius_rational} follows. Also,  by the second part of \Cref{thm:game_value_rational_computable},
there exists an algorithm that computes $V(G,H,P)$ from $G,H,P$ in finite time. The definitions of  $H$ and $P$ obviously show how the can be computed from the labeled graph $\mathcal{G}$, so \Cref{thm:covering_radius_compute} also  follows. 
\end{proof}

We now introduce some more notation in preparation for the proof of \Cref{thm:game_value_rational_computable}. Given $p \in \cW_n(H)$ and $v,w \in \cV{G}$ denote
\[
M (p , v \to w) := \inf_{q \in \cW_n(G,v \to w)}\sum_{i=1}^n P(q_i,p_i),
\]
where we adopt the convention that $\inf \emptyset = + \infty$. We denote by $M(p) \in \Z_\infty^{\cV{G} \times \cV{G}}$ the function 
\[
M(p) : \cV{G} \times \cV{G} \ni (v,w) \mapsto M(p, v \to w) \in \Z_\infty .
\]
We remark that the function $M(p)$ implicitly depends not only on the walk $p$ and the graph $H$, but also on the graph $G$ and the payoff function $P$, which we consider as being fixed in the background. Note that 
\begin{equation}\label{eq:V_G_H_via_M_p}
V_n(G,H,P)= \max_{p \in \cW_n(H)}\min M(p).    
\end{equation}

The following lemma explains the relevance of the tropical convolution operation to concatenation of walks:
\begin{lemma}\label{lem:tropical_conv}
For every $m,n \in \mathbb{N}$; $u,v,w \in \cV{H}$; $p \in \cW_n(H,u \to v)$; and $q \in \cW_m(H,v \to w)$ we have 
\[
M(p*q)=M(p)*M(q).
\]
\end{lemma}
\begin{proof}
For every $u,w \in \cV{H}$ and $m,n \in \N$ we have 
\begin{equation}\label{eq:cW_n_plus_m}
\cW_{n+m}(H,u \to w) = \biguplus_{v \in \cV{H}}\left\{ q*q' ~:~ q \in \cW_n(H,u \to v),~ q' \in \cW_m(H,v \to w)\right\}.
\end{equation}
Hence if    $u,w \in \cV{G}$, $p \in \cW_n(H)$ and $p' \in \cW_m(H)$ with $p$ terminating at the initial vertex of $p'$, then
\[
M(p*p',u \to w) = \min_{v \in \cV{G}}\inf_{q \in \cW_n(G,u \to v)} \inf_{q' \in \cW_m(G,v \to w)}\left( \sum_{i=1}^nP(q_i,p_i) + \sum_{i=1}^mP(q'_{i},p'_i)\right).
\]
It follows that
\[
M(p*p',u \to w) = \min_{v \in \cV{G}} \left[M(p,u \to v)+M(p',v \to w)\right].
\]
We thus have
\[
M(p*p')=M(p)*M(p').
\]
\end{proof}

Given $v,w \in \cV{H}$, set  
\begin{align*}\label{eq:M_n_def}
\cM_n(v\to w)   & : = \{ M(p)~:~ p \in \cW_n(H,v \to w)\} , \\ \cW_n^\dagger(H,v \to w) & : = \left\{ p \in \cW_n(H,v \to w)~:~ M(p) \in \cM_n(H,v \to w)^\dagger \right\} .
\end{align*}
For notational convenience, for $v,w \in \cV{H}$ and $n \in \N$  we also write
\begin{align*}
\cM_n^\dagger(v \to w)  & : = \left( \cM_n(v \to w)\right)^\dagger, & \overline{\cM}_n^\dagger(v \to w) & : = \cM_n^\dagger(v \to w) - V_n(G,H,P). 
\end{align*}

Also, given $v \in \cV{H}$,  let $\cW_n^\dagger(H,v):= \bigcup_{w \in \cV{H}}\cW_n^\dagger(H,v \to w)$ and $\cW_n^\dagger(H):=\bigcup_{v \in\cV{H}}\cW_n^\dagger(H,v)$.
We refer to the elements of $\cW_n^\dagger(H)$ as \emph{non-improvable walks} of length $n$.

The relevance of $\cM_n^\dagger(v \to w)$ to the computation of $V_n(G,H,P)$ is expressed by the following simple lemma:
\begin{lemma}\label{lem:V_max_M_star}
    \begin{equation*}\label{eq:V_n_via_M_n_star}
    V_n(G,H,P)= \max_{u,v \in \cV{H}}\max_{f \in \cM_n^
\dagger(u \to v)} \min f.    
    \end{equation*}
\end{lemma}
\begin{proof}
From \eqref{eq:V_G_H_via_M_p} it follows directly that
\[
V_n(G,H,P) = \max_{u,v \in \cV{H}}\max_{f \in \cM_n(u \to v)} \min f.  
\]
Let $\cM_n^-(u \to v) := \cM_n(u \to v) \setminus \cM_n^\dagger(u \to v)$, then
\[
V_n(G,H,P) = \max_{u,v \in \cV{H}}\max \left\{ \max_{f \in \cM_n^-(u \to v)} \min f, \max_{f \in \cM_n^\dagger(u \to v)} \min f\right\}.  
\]
Because $\cM_n(u \to v)$ is a finite set, for every $f \in \cM_n^-(u \to v)$ there exists $f^\dagger \in \cM_n^\dagger(u \to v)$ such that $f < f^\dagger$, and in particular $\min f < \min f^\dagger$. It follows that  $\max_{f \in \cM_n^-(u \to v)} \min f < \max_{f \in \cM_n^\dagger(u \to v)} \min f$.
From this, the result follows directly.
\end{proof}

\begin{lemma}\label{lem:M_p_finite}
    Under the assumption that $G$ is primitive, there exists $n_0 \in \N$ such that for every $n \geq n_0$ and $p \in \cW_n(H)$ we have $M(p) \in \Z^{\cV{G} \times \cV{G}}$, meaning that the function $M(p)$ does not take the value $+\infty$.
     
\end{lemma}
\begin{proof}
        It is well known (e.g., see~\cite[Theorem 4.5.8]{LindMarcus}) that an irreducible graph $G$ is primitive if and only if there exists $n_0 \in \N$ such that for any two vertices $v,v'\in \cV{G}$ and any $n \ge n_0$ there exists a walk of length $n$ from $v$ to $v'$.
          In other words, there exists $n_0$ such that for every $n \ge n_0$, we have $\cW_n(G,v \to w) \ne \emptyset$. Hence, for $p \in \cW(H)$ with $n \geq n_0$, we have that $M(p,v\to v')$ is a finite integer, so $M(p) \in \Z^{\cV{G} \times \cV{G}}$.
\end{proof}

As a direct consequence of \Cref{lem:star_trop_conv} and \Cref{lem:tropical_conv} we conclude that any segment of a non-improvable walk is also non-improvable:
\begin{lemma}\label{lem:non_improvable_subpath}
    For any $m,n \ge \N$ and any $p \in \cW_n(H)$ and $q \in \cW_m(H)$ if $p*q \in \cW_{n+m}^\dagger(H)$
    and $M(p*q)$ is $\Z$-valued,
    then necessarily $p \in \cW_{n}^\dagger(H)$ and $q \in \cW_m^\dagger(H)$.
\end{lemma}

\begin{proof}
By \Cref{lem:tropical_conv}, we have that 
\[M(p*q)=M(p)*M(q) \in \cM_n(H,u \to v)*\cM_m(H,v \to w).\]
The result follows from \Cref{lem:maximal_factorization}.
\end{proof}

For $p \in \cW_n(H)$ write
    \begin{align*}
         \max M(p)  & := \max_{v,w \in \cV{G}}M(p,v \to w) ,    & \min M(P) & := 
    \min_{v,w \in \cV{G}}M(p,v \to w).
    \end{align*}
   
\begin{lemma}\label{lem:M_n_recursion}
For every $m,n \in \N$ and $u,w \in \cV{H}$, the following identity holds:
\begin{equation}\label{eq:M_n_recursion}
\cM_{n+m}^\dagger(u \to w) = \left(\bigcup_{v \in \cV{H}}\cM_n^\dagger(u \to v)*\cM_m^\dagger(v \to  w)\right)^\dagger . 
\end{equation}

\end{lemma}
\begin{proof}

It is immediate that
\[
\cM^\dagger_{n+m}(u \to w) = \{ M(p) ~:~  p \in \cW_{n+m}(H,u \to w)\}^\dagger.
\]
Using \eqref{eq:cW_n_plus_m} we thus have
\[
\cM^\dagger_{n+m}(u \to w) = \left(\bigcup_{v \in \cV{H}}\{ M(q*q') ~:~ q \in \cW_n(H,u \to v),~ q' \in \cW_m(H,v \to w)\}\right)^\dagger.
\]
So using \Cref{lem:tropical_conv} and \Cref{lem:star_trop_conv}: 
\[
\cM^\dagger_{n+m}(u \to w) = \left(\bigcup_{v \in \cV{H}}\left(\cM_n^\dagger(u \to v) * \cM_m^\dagger(v \to w) \right)^\dagger\right)^\dagger.
\]
Finally, we conclude \eqref{eq:M_n_recursion} using the easily verified fact that for any finite collection of sets $A_1,\ldots,A_r \subseteq \Z_\infty^{\cV{H}\times \cV{H}}$, we have that
\[ \left(\bigcup_{j=1}^r A_j^\dagger\right)^\dagger = \left( \bigcup_{j=1}^r A_j \right)^\dagger.\]
\end{proof}

\begin{lemma}\label{lem:M_p_variation_bounded}
    There exists $N \in \N$, and $C_0>0$ depending only on the graph $G$ such that for any $n >N$ and  $p \in \cW_n(H)$ 
we have:
    \begin{equation}\label{eq:M_p_variation_bound}
    \Delta M(p) \le C_0\|P\|.    
    \end{equation}
    
\end{lemma}
\begin{proof}
    By \Cref{lem:M_p_finite} there exists $n_0 \in \N$ so that
    for any $n \ge n_0$ and  any $p \in \cW_{n}(H)$,  $M(p)$ takes only values in $\Z$. Moreover, for any $n \geq n_0$ and $p \in \cW_n(H)$ and $v,w \in \cV{G}$ we have
    \[
   |M(p,v \to w)| \le n\|P\|.
    \]
    
    Now choose any  $N > 2n_0$, $n \ge N$ and  $p \in \cW(H)$. Then we can write $p=p'*p''*p'''$ with $p',p''' \in \cW_{n_0}(H)$ and $p'' \in \cW_{n-2n_0}(H)$.
    
    By \Cref{lem:tropical_conv} 
    \[
    M(p)=M(p')*M(p'')*M(p'').
    \]
    Because $p',p''' \in \cW_{n_0}(H)$, we have that $\left|M\left(p'\right)\right|,\left|M\left(p'''\right)\right| < n_0 \|P\|$. Applying \Cref{lem:3_prodcuct_delta} with $f=M(p),f_i=M\left(p^{(i)}\right)$ for $i = 1, 2, 3$,
    we get \eqref{eq:M_p_variation_bound} with $C=4n_0 \|P\|$.
\end{proof}

\begin{lemma}\label{lemma:non_imp_almost_optimal}
        There exist $N,C \in \N$ so that for any $n \ge N$ and any $u,v \in \cV{H}$
        \[
        \overline{\cM}_n^\dagger(u \to v) \subseteq \{-C,\ldots,C\}^{\cV{G} \times \cV{G}}.
        \]
\end{lemma}

\begin{proof}
By \Cref{lem:M_p_finite} there exists $n_0 \in \N$ so that for every $n \ge n_0$ and $p \in \cW_n(H)$, we have that
$M(p)$ is $\Z$-valued. Because the graph $H$ is also primitive, we can choose $n_0$ sufficiently big so that for any $n \ge n_0$, any two vertices in $H$ are connected by a walk of length $n$.

The proof of the lemma will be completed once we  prove that for $N := 3n_0 +1$ and some $C>0$ we have that for any $n  \ge N$, $p \in \cW_n^\dagger(H)$ and  $v,w \in \cV{G}$ we have
\begin{equation}\label{eq:M_minus_V_bounded}
-C \le M(p, v\to w) - V_{n-2n_0}(G,H,P) \le C.    
\end{equation}

Fix $n \ge N$.
There exist some $v^*,w^* \in \cV{H}$ such that $p^* \in \cW_{n-2n_0}(H,v^*,w^*)$.
Choose any $n \ge N$,  $v_0,w_0 \in \cV{H}$, $p \in \cW_n^\dagger(H, v_0 \to w_0)$  and $v,w \in \cW(G)$.
Let
\begin{align*}
p'  & = (p_1,\ldots,p_{n_0}) \in \cW_{n_0}(H, v_0 \to v_1) , \\
p'' & =(p_{n_0+1},\ldots,p_{n-n_0}) \in \cW_{n - 2n_0}(H, v_1 \to w_1), \\
p'''    & =(p_{n-n_0+1},\ldots ,p_n) \in \cW_{n_0}(H, w_1 \to w_0) ,
\end{align*}
where $v_1 = t\left(p_{n_0}\right) = s\left(p_{n_0+1}\right)$ and $w_1=t\left(p_{n-n_0}\right)=s\left(p_{n-n_0+1}\right)$.

Then by definition $p=p'*p''*p'''$, where
$p',p''' \in \cW_{n_0}(H)$ are the initial and terminal segments of length $n_0$ and $p'' \in \cW_{n-2n_0}(H)$ is the ``middle segment'' of $p$.
Find $v^*,w^* \in \cV{G}$ such that $(v^*,w^*)$ is a minimizer of $M\left(p''\right)$, i.e.
\[
M\left(p'',v^* \to w^*\right) =\min  M\left(p''\right).
\]
By definition of $M(p,v \to w)$, we have 
\[
M(p, v \to w) = \min_{q \in \cW_n(G,v \to w)}\sum_{i=1}^n P(p_i,q_i).
\]
Denote
\[
A(v,v^*)= \min_{q' \in \cW_{n_0}(G,v \to v^*)}\sum_{j=1}^{n_0} P(p_i,q'_i),
\]
\[
B(v^*,w^*) = \min_{q'' \in \cW_{n-2n_0}(G,v^* \to w^*)}\sum_{j=1}^{n-2n_0} P(p_{i+n_0},q''_i),
\]
and 
\[
C(w^*,w) = \min_{q''' \in \cW_{n_0}(G,w^* \to w)}\sum_{j=1}^{n_0} P(p_{n-n_0+i},q'''_i).
\]
Since 
\[
\cW_{n_0}(G,v \to v^*)*\cW_{n-2n_0}(G,v^* \to w^*)*\cW_{n_0}(G,w^* \to w) \subseteq \cW_n(G,v \to w) ,
\]
it follows that
\[
M(p,v \to w) \le A(v,v^*)+B(v^*,w^*)+C(w^*,w).
\]
Since $A(v,v^*) \le n_0\|P\|$, $B(v,w^*) = M(p'',v^* \to w^*) = \min M(p'')$ and $C(w^*,w) \le n_0\|P\|$, we have that:
By taking the minimum on the right hand side only over walks $q \in \cW_n(G,v \to w)$ for which $s\left(q_{n_0 + 1}\right) = v^*, s \left( q_{n - n_0 + 1}\right) = w^*$, we see that
\[
M(p, v \to w) \le 2n_0\|P\|+\min M\left(p''\right).
\]
By \eqref{eq:V_G_H_via_M_p}, we have that $\min M\left(p''\right) \le V_{n-2n_0}(G,H,P)$, so altogether we conclude that
\begin{equation}\label{eq:M_p_upper_bound}
M(p, v \to w) \le 2n_0\|P\|+V_{n-2n_0}(G,H,P).    
\end{equation}

By \eqref{eq:V_G_H_via_M_p}, there exists $v_2,w_2 \in \cV{H}$ and $p^* \in \cW_{n-2n_0}(H,v_2 \to w_2)$ such that
\[
V_{n-2n_0}(G,H,P)=\min M(p^*).
\]
Since we chose $n_0$ to be sufficiently large, there exist $p^a \in \cW_{n_0}(H,v_0 \to v_2)$ and $ p^b \in \cW_{n_0}(H,w_2 \to w_0)$.
Let $\tilde p = p^a p^* p^b \in \cW_{n}(H,v_0 \to w_0)$. Because $p \in \cW^\dagger_n(H,v_0 \to w_0)$, it follows that there exist $v',w' \in \cV{G}$ so that

\[
M(\tilde p,v'\to w') \le M(p,v' \to w').
\]
On the other hand we have that
\[
M(p,v' \to w') \le M(p,v \to w) + \Delta M(p)\le M(p,v \to w) +  C_0\|P\|,
\]
where $C_0>0$ is a constant as  in \eqref{eq:M_p_variation_bound} in \Cref{lem:M_p_variation_bounded}.
Because $\tilde p=p^a*p^**p^b$ we have
\[M(\tilde p,v' \to w') \ge \min p^* - 2n_0\|P\| = V_{n-2n_0}(G,H,P) - 2 n_0 \|P\|.\]
Combining the above inequalities, we see that 
\begin{equation}\label{eq:M_p_lower_bound}
M(p, v\to w) \ge V_{n-2n_0}(G,H,P) - \left(2n_0+C_0\right)\|P\|.    
\end{equation}

Combining \eqref{eq:M_p_upper_bound} and \eqref{eq:M_p_lower_bound} we conclude that \eqref{eq:M_minus_V_bounded} holds with $C > \left(2n_0+C_0\right)\|P\|$.
\end{proof}

\begin{lemma}\label{lem:M_star_periodic}
    There exist $n_1,k \in \N$ and $V^\dagger \in \Z$ such that
    \begin{equation}\label{eq:M_n_star_ev_periodic}
    \forall v,w \in \cV{H}~ \forall n \geq n_1 ~ \left[ \cM_{n+k}^\dagger (v \to w)  = \cM_{n}^\dagger (v \to w) + V^\dagger \right].     
    \end{equation}
    and
    \begin{equation}\label{eq:V_n_eventually_periodic}
    V_{n+k}(G,H,P)=V_n(G,H,P)+V^\dagger.    
    \end{equation}
    Furthermore, if  $k \in \N$ and $V^\dagger \in \Z$ satisfy \eqref{eq:M_n_star_ev_periodic}, then
    \begin{equation}\label{eq:V_V_star_over_k}
      V(G,H,P) = \frac{1}{k}V^\dagger.
    \end{equation}
\end{lemma}
\begin{proof}
    Let $\mathcal{P}(\{-C,\ldots,C\}^{\cV{G}\times\cV{G}})$ denote the collection of subsets of functions on $\cV{G}\times\cV{G}$ that take values in $\{-C,\ldots,C\}$.
    By \Cref{lemma:non_imp_almost_optimal}, there exist $N \in \N$ and $C \in \N$ such that for any $n \ge N$  and $u,v \in \cV{H}$, we have that
    $(\overline{\cM}_n^\dagger(u \to v))_{(u,v) \in \cV{H} \times \cV{H}} \in \mathcal{P}(\{-C,\ldots,C\}^{\cV{G}\times\cV{G}})^{\cV{H}\times \cV{H}}$.
    Since the set 
    \begin{equation*}\label{eq:S_G_H_def}
    S_{G,H}:= \mathcal{P}(\{-C,\ldots,C\}^{\cV{G}\times\cV{G}})^{\cV{H}\times \cV{H}}
    \end{equation*}
    is finite, by the pigeonhole principle, there exists $N_1 , N_2 \in \N$  with $N \leq N_1 < N_2$ such that 
    \begin{equation}\label{eq:M_N_1_eq_M_N_2}
    \left(\overline{\cM}_{N_1}^\dagger(u \to v)\right)_{(u,v) \in \cV{H} \times \cV{H}}= \left(\overline{\cM}_{N_2}^\dagger(u \to v)\right)_{(u,v) \in \cV{H} \times \cV{H}}.
    \end{equation}
    Let $k=N_2-N_1$. By \Cref{lem:M_n_recursion}, together with \eqref{eq:set_conv_plus_const}, we see that there exists a function $\Psi:S_{G,H} \to S_{G,H}$ such that for every $n \ge N$, we have
    \[
    \left(\overline{\cM}_{n+1}^\dagger(u \to v)\right)_{(u,v) \in \cV{H} \times \cV{H}}=
    \Psi\left( \left(\overline{\cM}_{n}^\dagger(u \to v)\right)_{(u,v) \in \cV{H} \times \cV{H}}\right).
    \]
    Hence by \eqref{eq:M_N_1_eq_M_N_2} we see that for every $n \ge N_1$
    \[
     \left(\overline{\cM}_{n+k}^\dagger(u \to v)\right)_{(u,v) \in \cV{H} \times \cV{H}} = \left(\overline{\cM}_{n}^\dagger(u \to v)\right)_{(u,v) \in \cV{H} \times \cV{H}}.
    \]
    By definition 
    \[
    \cM_n(u \to v) = \overline{\cM}_{n}^\dagger(u \to v) +V_n(G,H,P),
    \]
    This means that there exists $V^\dagger \in \Z$ such  that  \eqref{eq:M_n_star_ev_periodic} and \eqref{eq:V_n_eventually_periodic} hold for every  $n \ge N_1$.
    From \eqref{eq:V_n_eventually_periodic} we see that for every $n \ge N_1$ and $\ell \in \N$
    \[
    V_{n+\ell k}(G,H,P) = V_n(G,H,P) + \ell V^\dagger.
    \]
    Dividing by $n+\ell k$ and taking $\ell \to \infty$ we see that $\lim_{m \to \infty}\frac{1}{m}V_m(G,H,P)= \frac{1}{k}V^\dagger$, which proves \eqref{eq:V_V_star_over_k}.
\end{proof}

\begin{proof}[Proof of \Cref{thm:game_value_rational_computable}]
This is immediate from our proof of \Cref{lem:M_star_periodic}. In particular, we can compute $\overline{\cM}_n^\dagger$ for $n > N$ until we find $n_2 > n_1 > N$ such that $\overline{\cM}_{n_1}^\dagger = \overline{\cM}_{n_2}^\dagger$, and take $V^\dagger = V_{n_2}(G, H, P) - V_{n_1}(G, H, P)$.

\end{proof}

\section{The space of non-improvable paths}
\label{sec:non_improvable_paths}

As in \cite{meyerovitch2025nonalternatingmeanpayoffgames}, the value of  $V(G,H,P)$ can be regarded as the Nash equilibrium payoff in the following
two-player game:
\begin{itemize}
    \item Alice chooses a bi-infinite walk $p=(p_n)_{n \in \Z}$ in the graph $H$
    \item Bob observes the walk $p$ chosen by Alice and chooses a bi-infinite walk $q=(q_n)_{n \in \Z}$ in $G$.
    \item Bob pays Alice $\liminf_{n \to \infty}\frac{1}{n}\sum_{j=1}^nP(q_j,p_j)$.
\end{itemize}
In this section we will investigate the space of ``optimal strategies'' for Alice and Bob in a certain sense.

Given a graph $G$, we define 
    \[
    X_G := \{ x\in \cE{H}^\Z~:~ \forall n \in \Z \; \left( t(x_n)=s(x_{n+1})\right)\}.
    \]
to be the \emph{edge shift} corresponding to $G$ (c.f. \cite[Definition 2.2.5]{LindMarcus}). The shift space $X_G$ is a shift of finite type, and in particular sofic.

The quantity $V(G,H,P)$ defined by \eqref{eq:V_def} is the Nash equilibrium value of the game in the sense that it is the maximal payoff that Alice can guarantee regardless of Bob's strategy, and also the minimal payoff that Bob can  guarantee regardless of Alice's strategy.
The following result, 
which is a particular case of \cite[Theorem 2.10]{young2025adversarialergodicoptimization}, shows how to  express the value $V(G,H,P)$ using shift-invariant probability measures on $X_H$ and $X_G$. 
The following result follows directly from \cite[Theorem 2.10]{young2025adversarialergodicoptimization}:
\begin{thm}\label{thm:V_G_H_P_via_measures}
Let $G,H$ be primitive graphs and $P:\cE{G} \times \cE{H} \to \Z$. Then 
    \begin{equation}
        V(G,H,P) = \sup_{\mu \in \mathrm{Prob}_\sigma(X_H)}\inf_{\lambda \in \Prob_\sigma(X_H \times X_G) \cap (\pi_H)_*^{-1}(\{\mu\})}\int P(x_0,y_0)d\lambda(x,y),
    \end{equation}
where $\Prob_\sigma(X_H)$ (resp. $\Prob_\sigma(X_H \times X_G)$) denotes the family of shift-invariant Borel probability measures on $X_H$ (resp. $X_H \times X_G$).
\end{thm}

We now describe certain shifts that in a certain well-defined sense capture all the ``optimal strategies'' of Alice and Bob.
Consider primitive graphs $G,H$
and $P:\cE{G} \times \cE{H} \to \Z$.
For $p \in \cW_n(H)$,
let 
\[
\cR(p) : = \bigcup_{v,w \in \cV{G}}\left\{ q \in \cW_n(G,v \to w)~:~ M(p,v \to w) =\sum_{j=1}^nP(q_j,p_j)\right\}.
\]
Define 
\begin{equation*}\label{eq:X_H_P_star_def}
    X_{H}^\dagger : = \left\{ x \in X_H ~:~ x_{[m,n]} \in \cW_{n-m+1}^\dagger(H) ~ \forall m <n \right\}.
\end{equation*}
and
\begin{equation}\label{eq:X_G_H_P_hat_def}
    Y_{(H,G)}^\dagger : = \left\{ (x,y) \in  X^\dagger_{H} \times X_G ~:~ y_{[m,n]} \in \cR(x_{[m,n]}) ~ \forall m <n \right\},
\end{equation}
where $x_{I}$ denotes the restriction of $x : \Z \to \cE{G}$ to the set $I \subseteq \Z$.

It is easy to verify that both $X_H^\dagger$ and 
$Y_{(H,G)}^\dagger$ are closed and shift-invariant. 
The following lemma shows that they are non-empty.
\begin{lemma}\label{lem:X_H_dagger_Y_H_G_dagger}
    If $X_G$ and $X_H$ are non-empty, then
    both $X_H^\dagger$ and $Y_{(H,G)}^\dagger$ are non-empty. Moreover, for every $x \in X_H^\dagger$ there exists $y \in X_G$ such that $(x,y) \in Y_{(H,G)}^\dagger$.
\end{lemma}
\begin{proof}
    We first show that $X_H^\dagger$ is non-empty. Because $X_H$ is non-empty for every $n \in \N$, the set $\cW_{n}(H)$ is a non-empty finite set. So for every $n \in \N$   there exist non-improvable walks, meaning that $\cW^\dagger_n(H) \ne \emptyset$. We can then apply a compactness argument to find a point $x \in X_H^\dagger$.

    Now, using essentially the same arguments, we will show that for every $x \in X_H^\dagger$, there exists $y \in X_G$ such that $(x,y) \in Y_{(H,G)}^\dagger$.
    For every $n \in \N$ and  
    $p \in \cW_{n}(H)$ there exists
    $v,w \in \cV{G}$ 
    such that $\cW_n(G,v \to w) \ne \emptyset$. 
    By definition of $M(p,v \to w)$ as a minimizer, there exists $q \in \cW(G,v \to w)$ such $M(p,v \to w) = \sum_{j=1}^nP(q_j,p_j)$, that is $q \in \mathcal{R}(p)$. Also, if $p'$ is the sub-walk of $p$ obtained by removing the first and last edges, and $q'$ is the sub-walk of $q \in \cR(p)$ obtained by removing the first and last edges, then $q' \in  \mathcal{R}(p')$. Indeed, otherwise, if $q' \in \cW_{n-2}(G,v' \to w') \setminus \cR(p')$, then
    there exists $\tilde q'\in \cW_{n-2}(G,v'\to w')$ such that 
    \[
    \sum_{j=1}^{n-2}P(\tilde q'_j,p'_j)= M(p,v \to w) < \sum_{j=1}^{n-2}P(q'_j,p'_j),
    \]
    So by replacing the sub-walk $q'$ by $\tilde q'$ in $q$, we would obtain $\tilde{q} \in \cW_n(G)$ having the same initial and terminal vertices as $q$ such that
    \[
    \sum_{j=1}^n P(\tilde q_j,p_j) < \sum_{j=1}^n P(q_j,p_j),
    \]
    contradicting the assumption that $q \in \mathcal{R}(p)$.
    
    Fix $x \in X_H^\dagger$.
    For every $n \in \N$, choose a walk $\left( q^{(k)} \right)_{i = -k}^k \in \cR\left(x_{[-k, k]} \right)$. For each $k \in \N$, consider the point $y^{(k)} \in \left( \cE{G} \cup \{\theta \} \right)^\Z$ given by
    \begin{align*}
    y_i^{(k)}   & = \begin{cases}
    q_i^{(k)}   & \textrm{if $-k \leq i \leq k$,} \\
    \theta    & \textrm{otherwise.}
    \end{cases}
    \end{align*}
    Then by the compactness of $(\cE{G} \cup \{\theta\})^\Z$, the sequence $\left( y^{(k)} \right)_{k = 1}^\infty$ will have a limit point $y$, and we can verify directly that the limit point $y$ will satisfy $y \vert_{[m, n]} \in \cR\left( x_{[m, n]} \right)$ for all $m \leq n$. Thus $(x, y) \in Y_{(H, G)}^\dagger$.
\end{proof}
 
 The following proposition relates $V(G,H,P)$ with the shift-invariant probability measures on $X_H^\dagger$ and $Y_{(H,G)}^\dagger$. See section \cite{young2025adversarialergodicoptimization} for related statements.
 
\begin{thm}\label{thm:optimizing_measures}
Let $G,H$ be primitive graphs and $P:\cE{G} \times \cE{H} \to \Z$. Then
\begin{enumerate}
    \item  Let 
    $\mu$ be a shift-invariant probability measure on $X_H$.
    Then the following are equivalent:
    \begin{enumerate}[label = (1\alph*)]
        \item \[
    \mu\left(\left\{ x\in X_H~:~ \lim_{n \to \infty}\frac{1}{n}\min M(x_{[1,n]})=V(G,H,P) \right\}\right)= 1,
    \]
    \item 
     \[
    \limsup_{n \to \infty}\frac{1}{n}\int_{X_H}\min M(x_{[1,n]}) d\mu(x)= V(G,H,P),
    \]
    \item  $\mu \left( X^\dagger_H \right) =1$.
    \end{enumerate}
   \item Let $\lambda$ be a shift-invariant probability measure on $X_H^\dagger \times X_G$. Then the following are equivalent:
   \begin{enumerate}[label=(2\alph*)]
       \item \[
       \lambda \left( \left\{ (x,y) \in X_H^\dagger \times X_G ~:~ 
       \lim_{n \to \infty}\frac{1}{n}\sum_{i=1}^nP(y_i,x_i) = V(G,H,P)
       \right\}\right) =1 ,
       \]
       \item 
\[
\int_{X_H^\dagger \times X_G} P(x_0,y_0) d\lambda(x,y) = V(G,H,P).
\]
       \item $\lambda\left(Y_{(H,G)}^\dagger\right)=1$.
   \end{enumerate}
\end{enumerate}
\end{thm}

\begin{proof}
   \textbf{Proof of claim (1)}:

    \textbf{ (1c) $\Rightarrow$ (1a)}:
        By \Cref{lemma:non_imp_almost_optimal}, there exists $C \in \N$ such that
        \begin{align*}
        -C \le M(p)-V_n(G,H,P) \le C   &   & \textrm{for all $n \in \N, p \in \cW_n^\dagger(H)$,}
        \end{align*}
        and so for any $p \in \cW_n(H)$, we have
       \[
         M(p) \le V_n(G,H,P)+  C,
        \]
        and therefore for any $x^\dagger \in X_H^\dagger$, we have 
        \[
        \lim_{n \to \infty} \frac{1}{n} \min M\left(x^\dagger_{[1,n]}\right)=V(G,H,P).
       \]
        It also follows directly that for any $x \in X_H$ we have
        \[
        \limsup_{n \to \infty} \frac{1}{n} \min M(x_{[1,n]}) \le V(G,H,P).
        \]
        Let $g_n:X_H \to \R$ be defined by 
        \[g_n(x):= \frac{1}{n}\min M(x_{[1,n]}).\]
        Then by the above discussion, we have $\lim_{n \to \infty}g_n(x)=V(G,H,P)$ for every $x \in X_H^\dagger$, so $(c)$ implies that $g_n \stackrel{n \to \infty}{\to} V(G,H,P)$ $\mu$-almost surely. 
        We have thus proved (1c) $\Rightarrow$ (1a).

        \textbf{ (1a) $\Rightarrow$ (1b)}:
        Since $|g_n(x)| \le \|P\|$ for all $x \in X_H$, 
        the dominated convergence theorem implies that $\int g_n(x) d\mu(x) \stackrel{n \to \infty}{\to} V(G,H,P)$ if 
        $g_n(x) \stackrel{n \to \infty}{\to} V(G,H,P)$ $\mu$-almost surely.
        
        We have thus proved (1a) $\Rightarrow$ (1b).

        \textbf{ (1b) $\Rightarrow$ (1c)}:
        Suppose that (1b) holds.
        Fix any   $k \in \N; v,w \in \cV{H}$ and $p \in \cW_k(H, v \to w) \setminus \cW_k^\dagger(H, v\to w)$. Denote
        \[
        [p]:= \left\{ x \in X_H~:~ (x_{[1,k]}) = p \right\} \subseteq X_H.
        \]
        Our goal is to show that $\mu([p])=0$.
        
        Since $p \in  \cW_k(H, v \to w) \setminus \cW_k^\dagger(H, v \to w)$, 
        there exists $p' \in \cW_k^\dagger(H, v \to w)$ such that $  M(p) <  M(p')$.
        Let $\beta= \min \left(M(p')-  M(p)\right)$.
        Given $x \in X_H$ and $n \in \N$, define
        \[
        D_n(x,p) :=\left|\left\{1 \le j \le n -k ~:~ x_{[j,j+k]}=p  \right\}\right| =\sum_{j=1}^{n-k}\mathbbm{1}_{[p]}(\sigma^j(x)).
        \]
        We claim that for every $x \in X_H$, we have
        \begin{equation}\label{eq:M_le_V_minus}
        \min M(x_{[1,n]}) \le V(G,H,P) - \frac{\beta}{k}D_n(x,p).    
        \end{equation}
        
        Indeed, let $T:=\lfloor D_n(x,p)/k \rfloor$. Then we can find $t_1, \ldots, t_T$ such that
        \[1 \le t_1<t_1+k \le t_2 < t_2 + k \le t_3  \le \ldots < t_T \le n-k\]
        and $x_{[t_j + 1,t_j +k]}=p$ for all $1\le j \le T$.
        Let $x' \in X_H$ be the point obtained by 
        replacing each of the $T$ occurrences of $p$ in $x_{[1,n]}$ starting at $t_1,\ldots,t_T$ with an occurrence of $p'$.
        That is, $x'_{[t_j + 1,t_j+k]}=p'$ for $1\le j \le T$, and $x'_i = x_i$ for
        all $i \in \Z \setminus \bigcup_{j=1}^T[t_j + 1,t_{j+k}]$.
        Since $p',p$ both have the same vertex $v$ as their initial vertex and $w$ as their terminal vertex, $x'$ is indeed a valid point in $X_H$.

        It follows that
       \[
        \min M(x_{[1,n]}) \le \min M(x'_{[1,n]}) - \frac{\beta}{k}D_n(x,p)\le V_n(G,H,P)  - \frac{\beta}{k}D_n(x,p) - C/k,
        \]
        which proves \eqref{eq:M_le_V_minus}.
        Now for any $n > k$, using that $\mu$ is shift-invariant, we have
        \[
        \int_{X_H} D_n(x,p) d\mu(x)= \sum_{j=1}^{n-k}\int_{X_H} \mathbbm{1}_{[p]}(\sigma^j(x))d\mu(x) = (n-k) \mu([p]), 
        \]
        So:
    \[
    \lim_{n \to \infty}\frac{1}{n}\int \frac{\beta}{k}D_n(x,p)d\mu(x) = \frac{\beta}{k}\mu([p]).
    \]
    Let $f_n:X_H \to \R$ be given by
    \[
    f_n(x)= \|P\| -  \frac{1}{n}\min M(x_{[1,n]}).
    \]
    Since $f_n \ge 0$, Fatou's lemma together with \eqref{eq:M_le_V_minus} gives that
    \[
    \liminf_{n \to \infty} \int_{X_H} f_n(x) d\mu(x) \ge \|P\| - V(G,H,P)  +   \frac{\beta}{k}\mu([p]).
    \]
    Rearranging terms, we conclude that 
    \[
    \limsup_{n \to \infty} \frac{1}{n}\int_{X_H} \min M(x_{[1,n]}) d\mu(x) \le V(G,H,P) - \frac{\beta}{k}\mu([p]).
    \]
    By the assumption that (1b) holds, since $\beta>0$ we have that $\mu([p])=0$.
    Since $k \in \N$, $p \in \cW_k(H) \setminus \cW^\dagger(H)$ were arbitrary, and
    \[X_H^\dagger = X_H \setminus \bigcup_{k \in \N}\bigcup_{n \in Z}\bigcup_{p \in \cW_k(H) \setminus \cW^\dagger(H)} \sigma^n([p]),\]
    we conclude that $\mu(X_H^\dagger)=1$ so (1c) holds.
    
       \textbf{Proof of claim (2)}:

\textbf{(2c) $\Rightarrow$ (2a)}: Let $\pi_H:X_H^\dagger \times X_G \to X_H^\dagger$ be given by $\pi_H(x,y)=x$, and suppose that (2c) holds. Note that in light of \Cref{lem:X_H_dagger_Y_H_G_dagger}, this $\pi_H$ is in fact a factor map of shifts. Let $\mu:= \left( \pi_H\right)_* \lambda$. Then $\mu$ is a shift-invariant probability measure on $X_H^\dagger$, thus satisfying condition (1c). By the implication (1c) $\Rightarrow$ (1a) from claim (1), we see
        that $\lim_{n \to \infty}\frac{1}{n}\min M(x_{[1,n]})= V(G,H,P)$ $\mu$-almost surely. By definition of $Y_{(H,G)}^\dagger$,
        for any $(x,y) \in Y_{(H,G)}^\dagger$ we have that 
       \[
       \frac{1}{n}\sum_{j=1}^n P(y_j,x_j) = \frac{1}{n} M(x_{[1,n]}, s(y_1) \to t(y_n)).
       \]
       It follows that $\lim_{n \to \infty}\frac{1}{n}\min \frac{1}{n}\sum_{j=1}^n P(y_j,x_j)= V(G,H,P)$ $\lambda$-almost surely, so (2a) holds.

        \textbf{(2a) $\Rightarrow$ (2b)}: Suppose that (2a) holds.
        Let $F:X_H^\dagger \times X_G \to \R$ be given by $F(x,y)= P(x_0,y_0)$.
        Note that because $\lambda$ is shift-invariant we have that for every $n \in \N$
        \begin{align*}
        \int_{X_H^\dagger \times X_G}P(x_0,y_0)d\lambda(x,y)    & = \frac{1}{n}\int_{X_H^\dagger \times X_G}\sum_{j=1}^n F(\sigma^j(x),\sigma^j(y))d\lambda(x,y) \\
            & =\int_{X_H^\dagger \times X_G} \frac{1}{n} \sum_{j = 1}^n P(y_j,x_j)d\lambda(x,y).
        \end{align*}
        By (2a), the integrand on the right hand side tends to $V(G,H,P)$ $\lambda$-almost surely, so (as in the proof of the implication (1a) $\Rightarrow$ (1b) in claim (1)) we get (2b).

        \textbf{(2b) $\Rightarrow$ (2c)}: The proof idea is again similar to the corresponding proof of the implication (1b) $\Rightarrow$ (1c) in claim (1).

        Assume that (2b) holds. Fix any $k \in \N; p \in \cW_k^\dagger(H); u, w \in \cV{G}$ and $q \in \cW(G, u \to w) \setminus \cR(p)$. Our goal is to show that
        $\lambda ([p,q])= 0$, where
        \[
        [p,q] = \left\{
        (x,y) \in X_H^\dagger \times X_G~:~ x_{[1,k]}=p,~ y_{[1,k]}=q
        \right\}.
        \]
        Let $\beta: = \sum_{i=1}^kP(q_i,p_i) - M(p,u \to w)$. Because $q \not \in \cR(p)$, we have that $\beta >0$.

        Given $(x,y) \in X_H^\dagger \times X_G$ let 
        \[
        D_n(x,y,p,q) := \sum_{j=0}^{n-k-1}\mathbbm{1}_{[p,q]}(\sigma^j(x),\sigma^j(y)).
        \]
        As in the corresponding part of our proof of (1b)$\Rightarrow$(1c), we see that 
        \[
        \sum_{j=0}^{n-k-1}P(y_j,x_j) \ge V_n(G,H,P) + \frac{\beta}{k}D_n(x,y,p,q) -C.
        \]
        Since $\int D_n(x,y,p,q) d\lambda(x,y) = (n-k)\lambda([p,q])$,
        dividing by $n$, integrating, and taking $n \to \infty$ we see
        that
        \[
        \int_{X_H^\dagger \times X_G}P(x_0,y_0)d\lambda(x,y) \ge V(G,H,P) -\frac{\beta}{k}\lambda([p,q]).
        \]
        So assuming (2b), we conclude that $\lambda([p,q])=0$ whenever $q \not \in \cR(p)$. This shows that $\lambda(Y_{(H,G)}^\dagger)=1$, i.e. (2c) holds.

\end{proof}

\begin{thm}\label{thm:X_H_star_sofic}
Both $X^\dagger_{H}$ and $Y^\dagger_{(H,G)}$ are sofic shifts.
\end{thm}
For the proof of \Cref{thm:X_H_star_sofic}, we will use a  sufficient condition for soficity.
We first recall a well-known  characterization of sofic shifts. Given a shift space $S \subseteq \Sigma^\Z$ over a finite alphabet $\Sigma$, and an admissible word $w \in \mathcal{C}(S)$, the \emph{follower set} of $w$ in $S$, which we denote by $F(S,w)$, is defined to be the set of all words $v \in \mathcal{C}(S)$ that can admissibly follows $w$. That is:
\[
F(S,w) := \left\{ v \in \mathcal{C}(S)~:~ wv \in \mathcal{C}(S) \right\}.
\]
A shift space $S$ is sofic if and only if it has a finite number of follower sets \cite[Theorem 3.2.10]{LindMarcus}, that is 
\[
| \{F(S,w)~:~ w \in \mathcal{C}(S)\}| < +\infty.
\]

Given a set of words $\mathcal{F} \subseteq \bigcup_{k \in\N} \Sigma^n$
and $n \in \N$ denote by 
$\mathcal{\hat C}(\mathcal{F}) \subseteq \Sigma^n$ the set of words 
in which no word from $\mathcal{F}$ occurs as a subword, i.e.
\begin{align*}
 \mathcal{\hat C}_n(\mathcal{F})    & := \left\{ w=(w_1,\ldots,w_n) \in \Sigma^n ~:~ (w_i,\ldots,w_j) \not \in \mathcal{F} ~\forall 1\le i \le j \le n \right\}, &
 \mathcal{\hat C}(\mathcal{F})  & : = \bigcup_{n \in \N} \mathcal{\hat C}_n(\mathcal{F}) .
\end{align*}
If $\mathcal{F}$ is a set of forbidden words that generates a shift $S \subseteq \Sigma^\Z$ (c.f.
\Cref{def:generator_of_shift}),
then we have $\mathcal{C}_n(S) \subseteq \mathcal{\hat {C}}_n(\mathcal{F})$; however, strict inclusion is possible.
In the literature, $\mathcal{\hat C}(\mathcal{F})$ is sometimes called ``the set of locally admissible words for $S$'' (suppressing the dependence on the set of forbidden words $\mathcal{F}$). 
A simple compactness argument
shows that if
$\mathcal{F}$ generates $S$
and $w \in \mathcal{\hat C}(\mathcal{F})$, then $v \in \mathcal{C}_n(S)$ if and only for every $k \in \N$ there exist $u,w \in \mathcal{\hat C}_k(\mathcal{F})$ such that $uvw \in \mathcal{\hat C}(\mathcal{F})$.
For $w \in \mathcal{\hat C}(\mathcal{F})$, we set
    \[
    \hat F(\mathcal{F},w) := \left\{ v \in \mathcal{\hat C}(\mathcal{F}) ~:~ wv \in \mathcal{\hat C}(\mathcal{F}) \right\}.
    \]
We refer to the set $\hat F(\mathcal{F},w)$ as the \emph{local follower set of $w$ with respect to $\mathcal{F}$}.

\begin{lemma}\label{lem:sofic_via_local_follower_sets}
  Let  $S \subseteq \Sigma^\Z$ be a shift space generated by a set $\mathcal{F} \subseteq \bigcup_{k \in\N} \Sigma^k$ of forbidden words, where $\Sigma$ is a finite alphabet.
    If the set of local follower sets $\{\hat F(\mathcal{F},w)~:~ w\in \mathcal{\hat C}(\mathcal{F})\}$ is finite, then $S$ is sofic.
\end{lemma}
\begin{proof}
    For a shift space $S \subseteq \Sigma^\Z$ and configuration $x \in S$, define the \emph{future set} of $x$ by
    \[
    F_\infty(S,x)=\left\{ (y_i)_{i = 0}^\infty \in \Sigma^{\Z_+}~:(\ldots, x_{-3}, x_{-2}, x_{-1}, y_0, y_1, y_2, \ldots)\in S \right\}.
    \]
    It is well known  that a shift space $S$ is sofic if and only if there are finitely many future sets (see \cite[exercise 3.2.8]{LindMarcus} for an explanation of how these future sets can be used to construct a labeled graph presentation of $S$), in the sense that the following holds:
    \[|\{ F_\infty(S,x)~:~ x\in S\}| < +\infty.\]
    Assume that $\mathcal{F}$ is a generator for $S$ and assume that the set of $\{\hat F(\mathcal{F},w)~:~ w\in \mathcal{\hat C}(\mathcal{F})\}$ is a finite set.
    Since $\mathcal{F}$
    genereates $S$,
    it follows that for any $x \in S$ we have
    \[
    F_\infty(S,x) = \left\{
    y_{[0,+\infty)} \in \Sigma^{\Z_+}~:~ y_{[0,n]}\in \hat F(\mathcal{F},x_{[-k,-1]}) ~\forall n,k \in \N
    \right\}.
    \]
    Since  $\hat F(\mathcal{F},x_{[-(k+1),-1]}) \subseteq \hat F(\mathcal{F},x_{[-k,-1]})$ for every $x \in S$ and $k \in \N$, and
    \linebreak$\left|\left\{\hat F(\mathcal{F},w)~:~ w\in \mathcal{\hat C}(\mathcal{F})\right\}\right| < +\infty$, it follows that for any $x \in S$ there exists $k \in \N$ such that
    \[
     F_\infty(S,x) = \left\{
    y_{[0,+\infty)} \in \Sigma^{\Z_+}~:~ y\vert_{[0, n]}\in \hat F(\mathcal{F},x\vert_{[-k,-1]}) ~\forall n \in \N
    \right\} = \hat F(\mathcal{F},x\vert_{[-k,-1]}).
    \]
    In particular, in this case the set of future sets of $S$ is a subset of the set of local follower sets of $\mathcal{F}$, and in particular it is finite, hence $S$ is sofic.
\end{proof}

\begin{proof}[Proof of \Cref{thm:X_H_star_sofic}]

We will prove that $X_{H}^\dagger$ and $Y_{(H,G)}^\dagger$ are sofic by proving that both have a finite number of local follower sets with respect to an appropriate generator, and applying \Cref{lem:sofic_via_local_follower_sets}.

We begin by proving that $X_{H}^\dagger$ has a finite number of local follower sets.
Observe that the following set of forbidden words generates $X_H^\dagger \subseteq \cE{H}^\Z$:
\[
\mathcal{F}_H^\dagger= \bigcup_{n \in \N}\left( \cE{H}^n \setminus \cW_n^\dagger(H) \right).
\]

By \Cref{lemma:non_imp_almost_optimal} there exists $C,N_1 \in \N$ such that $\overline{\cM}^\dagger_n(u \to v) \subseteq \{-C,\ldots,C\}^{\cV{G} \times \cV{G}}$ for every $u,v \in \cV{H}$ and $n \ge N_1$, so the set for every $u,v \in \cV{H}$ the set $\bigcup_{n \in \N}\overline{\cM}^\dagger_n(u \to v)$ is a finite set.
By \Cref{lem:M_star_periodic} there exists $N \ge N_1$ and $k \in \N$ such that 
   $\overline{\cM}^\dagger_{n+k}(u \to v)=\overline{\cM}^\dagger_n(u \to k)$ for every $n \ge N$.
   It follows that for any $n \ge N$ we have that $\mathcal{\hat C}_n(\mathcal{F}_H^\dagger)=\mathcal{W}_n^\dagger(H)$.
It's trivial that there are only finitely many local follower sets of the form $\hat{F} \left( \mathcal{F}_H^\dagger, p \right)$ for $p \in \bigcup_{n = 0}^{N - 1} \mathcal{C}_n(S)$, so it remains to be shown that there are only finitely many values that $\hat{F} \left( \mathcal{F}_H^\dagger, p \right)$ can take for $p \in \hat{C}_n \left( \mathcal{F}_H^\dagger \right), n \geq N$.

For $p \in \cW_n(H,u \to v)$, set
\[
\overline{M}(p) := M(p) - V_n(G,H,P).
\]
Then $p \in \cW_n^\dagger(H,u \to v)$ if and only if $\overline{M}(p) \in  \overline{\cM}^\dagger_n(u \to v)$.
Hence, for any $ n \ge N$, we have that $p \in \mathcal{\hat C}_n(\mathcal{F}_H^\dagger)$ if and only if there exist $u,v \in \cV{H}$ such that $p \in \mathcal{W}_n^\dagger(H, u \to v)$.
For $p \in \mathcal{W}_n^\dagger(H, u \to v)$, we have
\[
 \hat F(\mathcal{F}_H^\dagger,p) = \bigcup_{m \in \N}\bigcup_{w \in \cV{H}}\left\{ q \in \cW_m(H,v \to w)~:~ \overline{M}(p*q) \in \overline{\mathcal{M}}_{n+m}^\dagger(H) \right\}.
\]

For $n,m \in \N$ and $f,g \in \Z_{\infty}^{\cV{H} \times \cV{H}}$ denote
\[
\left( f * g\right)_{n,m} := f*g - V_{n+m}(G,H,P) +V_n(G,H,P)+V_m(G,H,P).
\]
It follows from \Cref{lem:tropical_conv} that
for any $n ,m \in \N$ 
\[
\overline{M}(p*q)=(\overline{M}(p)*\overline{M}(q))_{n,m}.
\]

We conclude that for $n \ge N$ and  $p \in \cW_n^\dagger(H,u \to v)$
\[
 \hat F(\mathcal{F}_H^\dagger,p) = \bigcup_{m \in \N}\bigcup_{w \in \cV{H}}\left\{ q \in \cW_m(H,v \to w)~:~ (\overline{M}(p)*\overline{M}(q))_{n,m} \in \overline{\mathcal{M}}_{n+m}^\dagger(H) \right\}.
\]
Also, since there exists a constant $V^\dagger \in \R$ such that $V_{n+k}(G,H,P) = V_n(G, H, P) + V^\dagger$ and $\overline{\cM}_n(H,u \to v)= \overline{\cM}_{n+k}(H,u \to v)$ for any $n \ge N$ (c.f. \Cref{lem:M_star_periodic}), we see that 
\[
(f*g)_{n+k,m}=(f*g)_{n,m} ~\forall n \ge N.
\]
From that it follows that for any $n \ge N; u,v \in \cV{H};$ and $p \in \cW^\dagger_{n+k}(H,u \to v)$, there exists $\tilde p \in \cW^\dagger_{n}(H,u \to v)$ such that $\overline{M}(p)=\overline{M}(\tilde p)$, and in that case we have that
$\hat F(\mathcal{F}_H^\dagger ,p)=\hat F(\mathcal{F}_H^\dagger
,\tilde p)$.
We thus conclude that any local follower set of $X_H^\dagger$ is of the form 
$\hat F(\mathcal{F}_H^\dagger,p)$ for some $n \le N+k, p \in \mathcal{C}_n(X_H^\dagger)$.
This shows that the set $\{ \hat F(\mathcal{F}_H^\dagger,p)~:~ p \in \cW_n(H), n \in \N\}$ of local follower sets is finite,  hence $X_H^\dagger$ is sofic.

We will now prove that $Y_{(H,G)}^\dagger$ is also sofic. The proof strategy is similar to the previous part.
Observe that the following is a generator for $Y_{(H,G)}^\dagger \subseteq \cE{H}^\Z\times\cE{G}^\Z$:
\[
\mathcal{F}_{(H,G)}^\dagger = \bigcup_{n \in  \N}\left\{ (p,q) \in \cE{H}^n\times \cE{G}^n~:~ p \not \in \mathcal{\hat C}(\mathcal{F}_H^\dagger) \mbox{ or } q \not \in \mathcal{R}(p)
\right\}.
\]
We will prove the soficity of $Y_{(H,G)}^\dagger$ by  applying \Cref{lem:sofic_via_local_follower_sets} and showing that \[\left|\{\hat F(\mathcal{F}_{(H,G)}^\dagger,(p,q))~:~ (p,q) \in \mathcal{\hat C}(\mathcal{F}_{(H,G)}^\dagger)\}\right| < +\infty.\] 

Hence for $n \ge N, m \in \N, (p,q) \in \mathcal{\hat C}_n(\mathcal{F}_{(H,G)}^\dagger)$, and 
$(p',q') \in \mathcal{\hat C}_m(\mathcal{F}_{(H,G)}^\dagger)$
such that
$p \in \cW_n^\dagger(H,u \to v)$ and $q \in \cW_n(G,u' \to v')$,
we have that $(p',q') \in \hat F(\mathcal{F}_{(H,G)}^\dagger,(p,q))$ if and only if there exists $w \in \cV{H}$ and $w' \in \cV{G}$ such that $p' \in \cW_m^\dagger(H,v \to w)$ and
\[M(p*p',u' \to w')=M(p,u' \to v')+
\sum_{j=1}^m P(q'_j,p'_j).\]
By unraveling the definitions, applying \Cref{lem:tropical_conv} and moving terms around, the last equation can be rewritten as follows:
\[
(\overline{M}(p)*\overline{M}(p'))_{n,m}(u' 
, w')-\overline{M}(p,u' \to v') +V_m(G,H,P) = \sum_{j=1}^m P(q'_j,p'_j).
\]
In particular, for every $n \ge N+k$, $p \in \cW_n^\dagger(H,u \to v)$ and $q \in \cW_n(G,u' \to v')$ there exist $\tilde p \in \cW_{n-k}^\dagger(H,u \to v)$, $\tilde q \in \cW_{n-k}(G,u' \to v')$ such that $\hat F(\mathcal{F}_{(H,G)}^\dagger,(p,q))=\hat F(\mathcal{F}_{(H,G)}^\dagger,(\tilde p,\tilde q))$.
Just like in the first part of the proof, this shows that the set of local follower sets $\{\hat F(\mathcal{F}_{(H,G)}^\dagger,(p,q))~:~ (p,q) \in \mathcal{\hat C}(\mathcal{F}_{(H,G)}^\dagger)\}$ is finite, hence $Y_{(H,G)}^\dagger$ is sofic.
\end{proof}

\begin{remark}
It is known that $X_H^\dagger$ can fail to be of finite type, as shown in \cite[Section 3.1]{young2025adversarialergodicoptimization}.
\end{remark}

As a corollary, we see that for any primitive  graphs and payoff function $P:\cE{G} \times \cE{H} \to \Z$ the Nash equilibrium of corresponding non-alternating  mean payoff game is obtained by periodic strategies in the following sense:
\begin{cor}\label{cor:periodic_Nash_equilibrium}
    Let $G$ and $H$ be primitive graphs and $P:\cE{G} \times \cE{H} \to \Z$. Then there exist 
    infinite walks $p \in \cW_{+\infty}(H) \subseteq \cE{H}^\N$, $q \in \cW_{+\infty}(G) \subseteq \cE{G}^\N$ that are both periodic in the sense that there exists $k \in \N$ such that $p_{n+k}=p_n$ and $q_{n+k}=q_n$ for all $n \in \N$, and so that 
    \begin{enumerate}
        \item $\lim_{n \to \infty}\frac{1}{n}\sum_{j=1}^nP(q_i,p_i) = V(G, H, P)$;
        \item for any $\tilde q \in \cW_{+\infty}(G)$ we have
        \[
        \liminf_{n \to \infty}\frac{1}{n}\sum_{j=1}^nP(q_i,\tilde p_i) \ge \lim_{n \to \infty}\frac{1}{n}\sum_{j=1}^nP(q_i, p_i);
        \]
        and
        \item for any $\tilde p \in \cW_{+\infty}(H)$ there exists $\tilde q \in \cW_{+\infty}(G)$ such that
        \[
        \limsup_{n \to \infty}\frac{1}{n}\sum_{j=1}^nP(\tilde q_i,\tilde p_i) \le \lim_{n \to \infty}\frac{1}{n}\sum_{j=1}^nP(q_i, p_i).
        \]
    \end{enumerate}
\end{cor}
\begin{proof}
    By \Cref{thm:X_H_star_sofic}, $Y^\dagger_{(H,G)}$ is a sofic shift.
    It is well-known that any sofic shift admits a periodic point, essentially because any infinite walk in a finite directed graph must contain a cycle. If $(x,y) \in Y^\dagger_{(H,G)}$ is a periodic point, then it follows directly from the definition of $Y^\dagger_{(H,G)}$ that $p= (x_1,\ldots,x_n,\ldots) \in \cW_{+\infty}(H), q= (y_1,\ldots,y_n,\ldots) \in \cW_{+\infty}(G)$ satisfy the assertions in the statement.
\end{proof}

\section{Concluding remarks and open questions}\label{sec:concluding_remarks}

In this article, we dealt only with primitive graphs. Some specific statements are known to fail for general (reducible) graphs: For instance the statement of \Cref{thm:V_G_H_P_via_measures} can fail if we do not assume that $G$ and $H$ are irreducible. 
See \cite{meyerovitch2025nonalternatingmeanpayoffgames} and \cite{young2025adversarialergodicoptimization} for examples of various pathologies exhibited in this context by non-primitive graphs.
However, it is likely that the primitivity assumption can be removed in the statement of \cref{thm:covering_radius_rational} and \Cref{thm:covering_radius_compute} through clever use of the periodic decomposition of irreducible graphs in order to break the problem into several primitive ``pieces," although this is beyond the scope of this article.

Another point worth addressing is the assumption that the payoff function $P$ takes rational values. There is no further difficulty generalizing to rational-valued payoff functions $P:\cE{G} \times \cE{H} \to \mathbb{Q}$, just by  multiplying by a common denominator. However, if we consider payoff functions that take real values, the same questions remain unsolved.

\begin{quest}\label{quest:comput_V_G_H_P}
Does there exist an algorithm that takes as input two primitive graphs $G, H$, and a payoff function $P : \cE{G} \times \cE{H} \to \Q(\sqrt{2})$, and outputs an ``exact value'' for $V(G, H, P)$ in finite time? Is $V(G,H,P) \in \Q(\sqrt{2})$?
\end{quest}

\begin{quest}
Let $G,H$ be primitive graphs, and $P:\cE{G} \times \cE{H}\to \R$ an arbitrary payoff function. Is the shift space $Y^\dagger_{(H,G)}$ defined by \eqref{eq:X_G_H_P_hat_def} a sofic shift? Does it always contain a periodic point?
\end{quest}

\bibliographystyle{amsplain}
\bibliography{lib}
\end{document}